\documentclass[12pt,reqno]{amsart}
\usepackage{graphicx}
\usepackage{epstopdf}
\usepackage{url}
\usepackage{booktabs}  
\usepackage{enumitem}
\usepackage{latexsym}
\usepackage{amscd, amsfonts,amsthm, amsmath,amssymb}
\newtheorem{theorem}{Theorem}[section]

\newtheorem{proposition}[theorem]{Proposition}

\newtheorem{lemma}[theorem]{Lemma}

\newtheorem{definition}[theorem]{Definition}

\newtheorem{corollary}[theorem]{Corollary}

\newtheorem{example}[theorem]{Example}

\def\S{\mathcal{S}}
\def\321{\underline{32}1}
\def\012{0\underline{12}}
\def\inv{\mathrm{inv}}
\def\des{\mathrm{des}}
\def\exp{\mathrm{exp}}

\def\rp{\mathrm{rp}}
\newcommand{\I}{\mathbf{I}}
\def\alert#1{{\color{blue!40!red}\textit{#1}}}
\newcommand{\qbinom}[2]{\genfrac{[}{]}{0pt}{}{#1}{#2}_q}

\usepackage{hyperref}
\hypersetup{
    colorlinks=true,
    linkcolor=cyan!50!blue,
    filecolor=magenta,      
    urlcolor=cyan,
    pdfpagemode=FullScreen,
    citecolor=green!60!black,
     }
\usepackage{cite}

\date{\today}
\usepackage{xcolor}%

\usepackage{tikz}
\usetikzlibrary{arrows.meta,positioning,calc,fit}
\usepackage{forest}
\usepackage{tikz}
\usetikzlibrary{trees,arrows.meta,positioning}
\usepackage{tikz}
\usetikzlibrary{positioning}
\usepackage{xcolor}
\tikzset{
  leaf/.style   ={circle,draw,minimum size=6mm,inner sep=1pt},
  minnode/.style={leaf, fill=blue!12, draw=blue!60!black},
  maxnode/.style={leaf, fill=red!12,  draw=red!60!black},
  note/.style   ={font=\scriptsize},
}
\usepackage{tikz}
\usepackage{forest}
\forestset{
  mytree/.style={
    for tree={
      draw, circle, inner sep=1.5pt, s sep=7mm, l sep=8mm,
      anchor=north
    }
  },
  highlight/.style={fill=gray!15, draw=black}
}

\begin{document}
\title[Inversion-descent enumerators of $\underline{32}1$-avoiding permutations]{Inversion-descent enumerators of $\underline{32}1$-avoiding permutations}
\author{Qiongqiong Pan}
\address{College of Mathematics and Physics, Wenzhou University\\
Wenzhou 325035, PR China}
\email{qpan@wzu.edu.cn}

\begin{abstract}
We give the recurrence relation for counting polynomials for permutations avoiding the $\321$  vincular pattern, with respect to inversion numbers and descent numbers, and consequently give the $q$-exponential generating function for these counting polynomials. Furthermore, we give another explanation for these counting polynomials in terms of restricted-growth words.
\end{abstract}
\maketitle

\section{Introduction}
Given permutations $\tau=\tau_1\tau_2\dots\tau_k\in\S_k$ and $\pi=\pi_1\pi_2\dots\pi_n\in\S_n$, we say that $\pi$ \alert{contains} $\tau$ if there exist indices $1\leq i_1<i_2<\dots<i_k\leq n$ such that the entries of the subsequence $\pi_{i_1}\pi_{i_2}\dots\pi_{i_k}$ are in the same relative order as the entries of $\tau$; otherwise  we say that $\pi$ \alert{avoids} $\tau$. We use $\S_n(\tau)$ to denote the subset of $\S_n$ in which the permutations avoid $\tau$. There are numerous fruitful results on pattern avoidance; see \cite{Bona12, Kitaev11} for references.

In this paper, we consider permutations that avoid a \emph{vincular pattern}, which was pioneered by Babson and Steingr\'imsson \cite{BS00}. A \alert{vincular pattern} $\tau$ is a permutation in $\S_k$, some of whose consecutive entries are underlined. If $\pi\in\S_n$ contains the vincular pattern $\tau$, and $\tau$ contains $\underline{\tau_i\tau_{i+1}\dots\tau_j}$, then the entries of $\pi$ corresponding to $\tau_i, \tau_{i+1},\dots,\tau_j$ must be adjacent.

Recently, Beveridge, Heysse and Robertson \cite{BHR26} determined the maximum inversion number for $\underline{32}1$-avoiding permutations and counted the number of permutations that achieve this maximum. Soon after, Beverdge, Hu, and Liu \cite{BHL26} showed that the set of permutations avoiding $\underline{32}1$ (resp. $3\underline{21}$) that achieve the maximum inversion number is enumerated by the Fibonacci numbers. Claesson \cite{C01} proved that the cardinality of the class $\S_n(1\underline{23})$ is the $n$-th Bell number. Then, by combining the reversal action on $\S_n(1\underline{23})$, it is obvious that the cardinality of $\S_n(\321)$ is also the $n$-th Bell number. In this paper, we discuss the inversion-descent polynomials on $\S_n(\321)$, and we give the recurrence relation for these polynomials in Section \ref{sect-1}. In Section \ref{sect-2}, we give a $q$-exponential generating function for the inversion-descent polynomials. In Section \ref{sect-3}, we relate these inversion-descent polynomials to restricted-growth words. Here, we give the definition of inversion number on permutations to end this section.
\begin{definition}\label{def-inversion}
The \alert{inversion number} of permutation $\pi\in\S_n$ is 
\[
\inv(\pi)=\{(i, j):i<j\;\text{and}\;\pi_i>\pi_j\}.
\]
\end{definition}

\section{Lehmer codes of $\S_n(\321)$}\label{sect-1}
In this section, we first recall the definition of Lehmer codes, which is a way of encoding permutations in combinatorics, and then give some results on Lehmer codes of $\S_n(\321)$.
\begin{definition}
The collection $\mathcal{L}_n$ of \alert{Lehmer codes} of length $n$ is
\[
\mathcal{L}_n=\{(p_1, p_2,\dots, p_n): 0\leq p_i\leq n-i\;\text{for}\;1\leq i\leq n\}.
\]
\end{definition}
For a permutation $\pi\in\S_n$, its Lehmer code is
$$
L(\pi):=(p_1, p_2,\dots, p_n) \;\text{where}\; p_i=|\{j>i: \pi_j<\pi_i\}|.
$$
Comparing the Definition \ref{def-inversion} with the Lehmer codes of permutations, we have that for $\pi\in\S_n$, $\inv(\pi)=w(L(\pi))$, where $w(L(\pi))=\sum_{i=1}^np_i$.

It is slightly more convenient for this paper to use reverse indexing. Thus, we give the following definition.
\begin{definition}\label{def-a}
 For $\pi\in\S_n$ and  let $L(\pi)=(p_1, p_2,\dots,p_{n-1}, p_n)$, then 
 $$
 a(\pi):=a_1a_2\dots a_{n-1}a_n,
 $$ 
 where 
 $a_i=p_{n+1-i}$ for $1\leq i\leq n$.
\end{definition}
Definition \ref{def-a} gives directly that for $\pi\in\S_n$ with $a(\pi)=a_1a_2\dots a_{n-1}a_n$, we have $a_1=0$ and $0\leq a_i\leq i-1$. Let $\I_n=\{e_1e_2\dots e_{n-1}e_n): 0\leq e_i\leq i-1\}$, which is also called the set of inversion sequences. The Lehmer code gives directly that $\{a(\pi): \pi\in\S_n\}=\I_n$. We also use $\I_n(\tau)$ to denote the subset of $\I_n$ whose sequences avoid the pattern $\tau$.

The following proposition was  proved by Frosini, Guerrini, and Rinaldi \cite{FGR25} as a corollary of  some other result. Here, for the sake of the overall flow of the paper, we will reprove this proposition directly.
\begin{proposition}\label{Pro-avoid}
We have $\pi\in\S_n(\321)$ if and only if $a(\pi)\in\I_n(0\underline{12})$.
\end{proposition}
\begin{proof}
If $\pi\in\S_n(\321)$ and $a(\pi)=a_1a_2\dots a_n$, then $a_1=0$. Suppose there exists an $i>1$ such that $a_1a_ia_{i+1}$ forms a $0\underline{12}$ pattern. This implies that $a_i<a_{i+1}$ and $\pi_{n+1-i}<\pi_{n-i}$. Since $a_{i}>0$, there exists a $j>n+1-i$ such that $\pi_{n+1-i}>\pi_j$. Thus, $\pi_{n-i}\pi_{n+1-i}\pi_j$ forms a $\321$ pattern, and this is a contradiction. So, $a(\pi)\in\I_n(0\underline{12})$.

If $\pi\notin\S_n(\321)$, that is, if there exist indices $1\leq i<j$ such that $\pi_i\pi_{i+1}\pi_j$ forms $\321$ pattern, then this implies that for $L(\pi)=(p_1, p_2,\dots, p_n)$, we have $p_i>p_{i+1}>0$, that is, $a_{n+1-i}>a_{n-i}>0$ in $a(\pi)=a_1a_2\dots a_n$. Then $a_1a_{n-i}a_{n+1-i}$ forms a $\012$ pattern, i.e., $a(\pi)\notin\I_n(\012)$.

Combining the above, we derive the result.
\end{proof}
The following proposition is essential for proving Theorem \ref{thm1}.
\begin{proposition}\label{pro-inversion}
For $\pi\in\S_n(\321)$ and $a(\pi)=a_1a_2\dots a_n$, whenever
$$
a_j<a_{j+1},
$$
we must have 
$$
a_j=0.
$$
\end{proposition}
\begin{proof}
Since $a_1=0$ and by Proposition \ref{Pro-avoid}, the result directly holds.
\end{proof}
We are now in a position to state our main result in this section. To that end, we first give some definitions.

We say $i\in [n-1]$ is a \alert{descent} of $\pi\in\S_n$ if $\pi_i>\pi_{i+1}$. We use $\des(\pi)$ to denote the number of descents of $\pi$. Define the following polynomials:
\[
I_n(q,t):=\sum_{\pi\in\S_n(\321)}q^{\inv(\pi)}t^{\des(\pi)}.
\]

 Now, the main theorem of this section is following.
 \begin{theorem}\label{thm1}
 For $n\geq1$, we have
 \begin{align}\label{eq-main}
 I_n(q,t)=\sum_{k=0}^{n-2}tq^{n-k-1}\qbinom{n-1}{k}I_k(q,t)+I_{n-1}(q,t),
 \end{align}
 with initial $I_0(q,t)=1$, where \[
\qbinom{n}{k}
=
\frac{(q;q)_n}{(q;q)_k(q;q)_{n-k}}
\]
is the Gaussian binomial coefficient and $(q;q)_n=\prod_{i=1}^n(1-q^i)$.
 \end{theorem}
\begin{proof}
Let $\pi\in\S_n(\321)$, then by Proposition \ref{Pro-avoid}, $a(\pi)=a_1a_2\dots a_n\in\I_n(\012)$. Suppose the last zero occurs at position $k+1$, where $0\leq k\leq n-1$; thus $a_{k+1}=0$. We divide $a(\pi)$ into two parts by the last zero:
\[
a(\pi)=
\underbrace{a_1a_2\cdots a_k}_{\text{prefix}}
0
\underbrace{a_{k+2}\cdots a_n}_{\text{suffix}}.
\]
Obviously, the prefix part of $a(\pi)$, namely $a_1a_2\dots a_k$, lies in $\I_k(\012)$, and by Proposition \ref{pro-inversion}, the suffix part of $a(\pi)$ is a weakly decreasing sequence with each $a_i>0$ for $k+1<i\leq n$. Now we analyze the contribution of the two parts to $I_n(q,t)$.
\begin{itemize}
\item [(A)]\textbf{Prefix part of $a(\pi)$}.
\\
Since $a_1a_2\dots a_k\in\I_k(\012)$ and $a_{k+1}=0$ does not create a descent with any element of the prefix, this part contributes $I_k(q,t)$.\\
\item[(B)]\textbf{Suffix part of $a(\pi)$}.
\\
Let $m=n-k-1$, then the suffix has length $m$. Because the entires of $a(\pi)$ satisfy $a_i\leq i-1$, the first suffix entry satisfies $a_{k+2}\leq k+1$. Since the suffix is weakly decreasing, this automatically guarantees the required upper bounds for all its subsequent entries. Consequently, the suffix is precisely a weakly decreasing sequence
$$
k+1\geq b_1\geq b_2\geq\dots\geq b_m\geq1.
$$
If  $k<n-1$, i.e., the suffix part is not empty, then this part contributes:
\begin{align}\label{eq-1}
\sum_{k+1\geq b_1\geq b_2\geq\dots\geq b_m\geq1}t\cdot q^{b_1+b_2+\dots+b_m},
\end{align}
where the factor $t$ is arises because $a_{k+1}=0<a_{k+2}$, which implies $\pi_{n-k-1}>\pi_{n-k}$. Moreover, $\pi_1<\pi_2<\dots<\pi_{n-k-1}$, since $a_{k+2}\geq a_{k+3}\geq\dots\geq a_{n}$.

Now, we compute \eqref{eq-1}.
Let $c_i=b_i-1$ for $m\geq i\geq1$. Then $k\ge c_1\ge c_2\ge\cdots\ge c_m\ge0$, and $b_1+\cdots+b_m=m+c_1+\cdots+c_m$.

Therefore, we have
\begin{align}\label{eq-3}
\sum_{k+1\ge b_1\ge\cdots\ge b_m\ge1}
q^{b_1+\cdots+b_m}
=
q^m
\sum_{k\ge c_1\ge\cdots\ge c_m\ge0}
q^{c_1+\cdots+c_m}.
\end{align}
The right side of \eqref{eq-3} is the standard generating function for partitions whose Ferrers diagrams fit inside an $m\times k$ rectangle. Hence
$$
\sum_{k\geq c_1\geq\cdots\geq c_m\geq0}
q^{c_1+\cdots+c_m}
=
\qbinom{k+m}{m}.
$$
Since $k+m=n-1$,
\begin{align*}
\sum_{k+1\geq b_1\geq b_2\geq\dots\geq b_m\geq1}t\cdot q^{b_1+b_2+\dots+b_m}=tq^{n-1-k}\qbinom{n-1}{k}.
\end{align*}
This completes the calculation of \eqref{eq-1}.
\end{itemize}
Now, combining cases (A) and (B) and adding the case when $k+1=n$ (i.e., suffix part is empty), we obtain \eqref{eq-main}.
\end{proof}
\begin{table}[htbp]
\centering
\caption{First few terms of $I_n(q,t)$}
\begin{tabular}{|c|l|}
\hline
$n$ & $I_n(q,t)$ \\ \hline
1  & $1$ \\
2   & $1+tq$ \\
3   & $1+2tq+2tq^2$ \\
4   & $1+3tq+4tq^2+t^2q^2+3tq^3+t^2q^3+tq^4+t^2q^4$ \\ \hline
\end{tabular}
\end{table}
In the following section, we will use Theorem \ref{thm1} to give a $q$-exponential generating function of $I_n(q,t)$.
\section{$q$-exponential generating function of $I_n(q,t)$}\label{sect-2}
Before calculating the $q$-exponential generating function of $I_n(q,t)$, we first give some basic definitions that will be used in this section.
\\
For $n\in\mathbb{N}$, let $0!_q=1$,
$$
[n]_q:=\frac{1-q^n}{1-q}\quad\text{and}\quad n!_q:=\prod_{k=1}^n[k]_q\quad\text{for}\quad n\geq1.
$$
Two classical $q$-analogues of the exponential function $e^x$ are
$$
\mathrm{exp}_q(x):=\sum_{n=0}^{\infty}\frac{x^n}{n!_q}\quad\text{and}\quad \mathrm{Exp}_q(x):=\mathrm{exp}_{1/q}(x)=\sum_{n=0}^{\infty}q^{n\choose2}\frac{x^n}{n!_q}.
$$
Let $\mathbf{R}$ be a ring with unity and characteristic zero. For $f(x)\in\mathbf{R}[[x]]$ the $q$-derivative is defined as
$$
D_q(f(x)):=\frac{f(qx)-f(x)}{(q-1)x}.
$$
Thus $D_q(1)=0$ and for $n>0$, $D_q(x^n)=[n]_qx^{n-1}$.
\\

Let 
$$
I(q,t;x):=\sum_{n=0}^{\infty}I_n(q,t)\frac{x^n}{n!_q}, 
$$
then the following result holds.
\begin{theorem}
We have
$$
I(q,t;x)=\prod_{k=0}^{\infty}\left([t(\exp_q(q^{k+1}x)-1)+1](q-1)q^kx+1\right)^{-1}.
$$
\end{theorem}
\begin{proof}
By Theorem \ref{thm1}, we have the following equation.
$$
\sum_{n=1}^{\infty}I_n(q,t)\frac{x^{n-1}}{(n-1)!_q}=t\sum_{n=1}^{\infty}\sum_{k=0}^{n-2}q^{n-k-1}I_k(q,t)\qbinom{n-1}{k}\frac{x^{n-1}}{(n-1)!_q}+\sum_{n=1}^{\infty}I_{n-1}(q,t)\frac{x^{n-1}}{(n-1)!_q},
$$
which yields
\begin{align}\label{generating}
\sum_{n=1}^{\infty}I_n(q,t)\frac{x^{n-1}}{(n-1)!_q}=t\sum_{n=1}^{\infty}\sum_{k=0}^{n-2}\frac{(qx)^{n-k-1}}{(n-1-k)!_q}I_k(q,t)\frac{x^k}{k!_q}+\sum_{n=1}^{\infty}I_{n-1}(q,t)\frac{x^{n-1}}{(n-1)!_q}.
\end{align}
By the definition of $q$-derivative and $\exp_q(x)$, equation \eqref{generating} becomes
\begin{align}\label{generating-2}
D_q(I(q,t;x))=t\cdot I(q,t;x)\cdot(\exp_q(qx)-1)+I(q,t;x).
\end{align}
Rewrite equation \eqref{generating-2} as
\begin{align*}
\frac{I(q,t;qx)-I(q,t;x)}{(q-1)x}=t\cdot I(q,t;x)\cdot(\exp_q(qx)-1)+I(q,t;x),
\end{align*}
which yields 
\begin{align}\label{generating-3}
\frac{I(q,t;qx)}{I(q,t;x)}=[t(\exp_q(qx)-1)+1](q-1)x+1.
\end{align}
Since $I(q,t;0)=1$, iterating the relation \eqref{generating-3} gives
$$
\frac{1}{I(q,t;x)}=\prod_{k=0}^{\infty}\left([t(\exp_q(q^{k+1}x)-1)+1](q-1)q^kx+1\right),
$$
which yields the result.
\end{proof}
\section{The connection with restricted-growth words}\label{sect-3}
In this section, we will give an explanation of $I_n(q,t)$ on restricted-growth words. The bridge is provided by using partitions via modifying a bijection introduced by Claesson \cite{C01}. 

A \alert{partition} of a set $S$ is a family $A=\{A_1, A_2,\dots, A_k\}$ of pairwise disjoint non-empty subsets of $S$ such that $S=\cup_iA_i$. We call $A_i$ a \alert{block} of $A$. The total number of partitions of $[n]$ is the $n$-th  Bell number. In \cite{C01}, Claesson gave a bijection between partitions of $[n]$ and $S_n(1\underline{23})$, and then combined it with the reversal map to prove the cardinality of $S_n(1\underline{23})$. Now, we modify this bijection, given by the following map $\Phi$, in order to adapt it to $S_n(\321)$.
\vskip 3mm
\noindent\textbf{Map $\Phi$.} Given a partition $A$ of $[n]$, a standard representation of $A$ by requiring that:
\begin{itemize}
\item[(a)] 
Each block is written with its least element last, write the other elements increasing.\\
\item[(b)]
The blocks are written in increasing order of their least element, and with dashes separating the blocks.
\end{itemize}
Define $\Phi(A)$ to be the permutation we obtain from $A$ by writing it in standard form and erasing the dashes.
\begin{example}\label{example-1}
As an illustration of $\Phi$, let
$$
A=\{\{1, 3, 6\}, \{2, 7\}, \{4\}, \{5, 8, 9\}\}.
$$
Its standard form is $361-72-4-895$, thus $\Phi(A)=361724895$.
\end{example}
A \alert{right-to-left minimum} of $\pi=\pi_1\pi_2\dots\pi_n\in\S_n$ is an element $\pi_i$ such that $\pi_i<\pi_j$ for every $j>i$. It is obviously that the last element of each block of $A$, when written in standard form, is a right-to-left minimum.

Now, we prove that the map $\Phi$ is indeed a bijection.
\begin{theorem}
For $n\geq1$, the map $\Phi$ is a bijection from partitions of $[n]$ to $\S_n(\321)$.
\end{theorem}
\begin{proof}
If $A$ is a partition of $[n]$, then we first prove that $\Phi(A)\in\S_n(\321)$. For convenience, let $\pi=\Phi(A)$. It is sufficient to prove that for every descent pair $(\pi_i, \pi_{i+1})$ (with $\pi_i>\pi_{i+1}$) of $\pi$, we have $\pi_j>\pi_{i+1}$ for all $j>i+1$. By the construction of the map $\Phi$, in the standard form of $A$, the only possible descent pairs occur at the last two elements of a block if that block has more than one elements. Since the last element of each block is the least element in its block, and the blocks are written in increasing order of their least element, all the elements after $\pi_{i+1}$ are larger than $\pi_{i+1}$. 

Now, we give the reverse of $\Phi$ to prove that $\Phi$ is a bijection. For $\pi\in\S_n(\321)$, since it avoids $\321$, the sequence between two adjacent right-to-left minimum elements is increasing.  We insert a dash before each right-to-left minimum element; this naturally gives a partition of $[n]$, and this procedure is obviously the reverse of $\Phi$.

Thus, we complete the proof.
\end{proof}
A word $\omega=w_1w_2\dots w_n$ is called a \alert{restricted-growth word of length $n$} if $w_1=1$ and, for every position $i>1$, the value $w_i$ is at most one more than the maximum value seen so far. 
This condition is written as $1\leq w_i\leq 1+\max\{w_1, w_2, \dots, w_{i-1}\}$. For instance, $\omega=1213$ is a restricted-growth word of length $4$. In fact, set partitions of $[n]$ have one-to-one correspondence with restricted-growth words of length $n$. For a partition $A=\{A_1, A_2, \dots, A_k\}$ of $[n]$, with $\min(A_1)<\min(A_2)<\dots<\min(A_k)$, where $\min(A_i)$ denotes the least element of $A_i$ for $k\geq i\geq 1$. Let $\pi=\Phi(A)$. Then  $\omega(\pi)=w_1w_2\dots w_n$ is the \alert{restricted-growth word of $\pi$}, where $w_i=j$ if and only if $i\in A_j$. This gives a correspondence from set partitions to restricted-growth words of length $[n]$. For Example \ref{example-1}, $\omega(\pi)=121341244$. Now, we have the following lemma.
\begin{lemma}\label{lem-3.1}
For $n\geq1$, let $A$ be a partition of $[n]$, and set $\pi=\Phi(A)$. Then we have
$$
\inv(\pi)=\inv(\omega(\pi))+n-k,
$$
where $k$ is the number of blocks of $A$.
\end{lemma}
\begin{proof}
The inversion number of $\inv(\pi)$ can be divided into two parts: one contributed by inversions inside a block in the standard form of $A$, and the other contributed by inversions between two different blocks of $A$. Assume $A_i$ is a block in the standard form of $A$; then the last element of $A_i$ (which is the least element of that block) forms an inversion with every other element of the block. Hence this block contributes $|A_i|-1$ inversions. Since $A$ has $k$ blocks, they contribute
$$
\sum_{i=1}^k(|A_i|-1)=n-k.
$$
Next, consider the inversions between two different blocks. Let $x<y$, with $x\in A_i$ and $y\in A_j$. Since the blocks occur from left to right according to their minima, an inversion between $x$ and $y$ in $\pi$ occurs precisely when $i>j$. But this is exactly the condition $w_x>w_y$ in $\omega(\pi)=w_1w_2\dots w_n$. Thus, the cross-block inversions in the standard form of $A$ are precisely the ordinary word inversions of $\omega(\pi)$.

Combining the two parts, we obtain the result.
\end{proof}
For example, let  $A=\{\{1, 4\}, \{2, 3\}\}$ be a set partition of $[4]$. Then $\pi=\Phi(A)=4132$ and $\omega(\pi)=1221$. We calculate $\inv(\pi)=4$, $\inv(\omega(\pi))=2$, and $n-k=4-2=2$. Hence, $\inv(\pi)=\inv(\omega(\pi))+n-k$.
The following lemma give the relationship between descents on $\S_n(\321)$ and the number of repeated elements in  restricted-growth words of length $n$. Before stating the lemma, we first give some definitions. Let $A$ be a partition of $[n]$, and let $\pi=\Phi(A)$. For $1\leq i\leq n$, we define 
$$
r_i(\omega(\pi))=\begin{cases}
0, &\text{if $i$ appears at most once in $\pi$}; \\
1, & \text{if $i$ appears more than once in $\pi$}.
\end{cases}
$$
Let $\rp(\omega(\pi))=\sum_{i=1}^nr_i(\omega(\pi))$.
\begin{lemma}\label{lem-3.2}
For $n\geq1$, let $A$ be a partition of $[n]$ and set $\pi=\Phi(A)$. Then we have
$$
\des(\pi)=\rp(\omega(\pi)).
$$
\end{lemma}
\begin{proof}
By the construction of $\Phi$, descents occur only at the last two elements of blocks that have length greater than $1$ in the standard form of $A$. Then the $i$-th block in the standard form of $A$ has length greater than $1$ if and only if  $i$ appears more than once in $\omega(\pi)$. This completes the proof.
\end{proof}
Now let $W_n$ be the set of restricted-growth words of length $n$, and for $w\in W_n$ let $\mathrm{L}(w)$ be the largest element in $w$. Combining Lemma \ref{lem-3.1} and Lemma \ref{lem-3.2}, we have the following theorem.
\begin{theorem}\label{theorem-3.2}
For $n\geq1$, we have
\begin{align*}
I_n(q,t)&=\sum_{w\in W_n}q^{\inv(w)+n-\mathrm{L}(w)}t^{\rp(w)}\\
&=q^n\sum_{w\in W_n}q^{\inv(w)-\mathrm{L}(w)}t^{\rp(w)}.
\end{align*}
\end{theorem}
\begin{proof}
The result holds due to the fact that for a partition $A$ of $[n]$, the number of blocks of $A$ is exactly the largest element of $\omega(\Phi(A))$.
\end{proof}
Now, let $J_n(q,t):=\sum_{w\in W_n}q^{\inv(w)-\mathrm{L}(w)}t^{\rp(w)}$. Then, by Theorems \ref{thm1} and \ref{theorem-3.2}, we have the following recurrence for $J_n(q,t)$.
\begin{corollary}
For $n\geq1$, we have
$$
q\cdot J_n(q,t)=t\sum_{k=0}^{n-2}\qbinom{n-1}{k}J_k(q,t)+J_{n-1}(q,t),
$$
with $J_0(q,t)=1$.
\end{corollary}

\end{document}